\documentclass[12pt]{amsart}
\usepackage{graphicx,amssymb}
\usepackage{longtable}

\usepackage{times}

\newtheorem{theorem}{Theorem}[section]
\newtheorem{proposition}[theorem]{Proposition}
\newtheorem{lemma}[theorem]{Lemma}

\theoremstyle{definition}
\newtheorem{definition}[theorem]{Definition}

\newcommand{\bdd}{\mbox{$\partial$}}

\newcommand{\inter}{\mbox{${\rm Int}$}}

\theoremstyle{remark}
\newtheorem{remark}[theorem]{Remark}

\numberwithin{equation}{section}

\begin{document}

\title[Twisted Alexander invariants and hyperbolic volume]
{Twisted Alexander invariants and hyperbolic volume} 

\date{\today}

\subjclass[2000]{Primary 57M27, Secondary 57M25}

\author{Hiroshi Goda}
\address{Department of Mathematics,
Tokyo University of Agriculture and Technology,
2-24-16 Naka-cho, Koganei,
Tokyo 184-8588, Japan}
\email{goda@cc.tuat.ac.jp}

\begin{abstract}
We give a volume formula of hyperbolic knot complements 
using twisted Alexander invariants.
\end{abstract}

\thanks{
This work was supported by JSPS KAKENHI 
Grant Numbers JP15K04868.
}

\keywords{twisted Alexander polynomial, hyperbolic knot, volume}

\maketitle

\section{Introduction}
The purpose of this note is to give a formula of the hyperbolic volume 
of a knot complement using twisted Alexander invariants. 

A twisted Alexander polynomial was first defined in \cite{lin} for knots in the 3-sphere, 
and Wada (\cite{wada})
generalized this work and showed how to define a twisted Alexander polynomial 
given only a presentation of a group and representations to $\mathbb Z$ and $\text{GL}(V)$ 
where $V$ is a finite dimensional vector space over a field. 
In \cite{kitano}, Kitano proved that in the case of knot groups the twisted 
Alexander polynomial can be regarded as a Reidemeiser torsion.

Let $M$ be a compact and oriented 3-manifold whose interior admits 
a finite volume hyperbolic structure. 
Porti (\cite{porti1}) has investigated the Reidemeister torsion of $M$ 
associated with the adjoint representation 
$\text{Ad}{\circ}\text{Hol}_M$ of its holonomy 
representation $\text{Hol}_M :\pi_1(M)\to\text{PSL}(2,\mathbb C)$,
and then Yamaguchi showed in \cite{yamaguchi1} a relationship between 
the Porti's Reidemeister torsion and the twisted Alexander invariant explicitly.

M\"{u}ller's work (\cite{muller}) provides the relation between the Ray-Singer 
torsion and the hyperbolic volume of a compact hyperbolic 3-manifold. 
By another work (\cite{muller0}) of M\"{u}ller on the equivalence 
between the Reidemeister torsion 
and the Ray-Singer torsion for unimodular representations, 
we know the hyperbolic volume of a compact 3-manifold can be expressed using 
a Reidemeister torsion. 
After the works, 
Menal-Ferrer and Porti (\cite{menal-porti2}) obtained a formula of the volume 
of a cusped hyperbolic 3-manifold $M$ using `Higher-dimensional Reidemeister torsion invariants', 
which are associated with representations $\rho_n:\pi_1(M)\to\text{SL}(n,\mathbb C)$
corresponding to the holonomy representation $\text{Hol}_M:\pi_1(M)\to\text{PSL}(2,\mathbb C)$
(see Section 3 for the detail).

In this note, we show that the Yamaguchi's method in \cite{yamaguchi0,yamaguchi1} 
is applicable to Higher-dimensional 
Reidemeister torsion invariants, so that we have  a formula of the hyperbolic volume 
of a knot complement using twisted Alexander invariants. 
Let $\Delta_{K,\rho_n}(t)$ be a twisted Alexander invariant of Wada's notation (\cite{wada}).
For the integer $k (>1)$,
set $\mathcal A_{K,2k}(t):=\frac{\Delta_{K,\rho_{2k}}(t)}{\Delta_{K,\rho_2}(t)}$ 
and $\mathcal A_{K,2k+1}(t):=\frac{\Delta_{K,\rho_{2k+1}}(t)}{\Delta_{K,\rho_3}(t)}$.

\begin{theorem}\label{thm:main}
Let $K$ be a hyperbolic knot in the 3-sphere. Then
$$
\lim_{k\to\infty}\frac{\log|\mathcal A_{K,2k+1}(1)|}{(2k+1)^2}
=
\lim_{k\to\infty}\frac{\log|\mathcal A_{K,2k}(1)|}{(2k)^2}
=
\frac{{\rm Vol}(K)}{4\pi}.
$$
\end{theorem}

In the last section, we give some calculations for the figure eight knot.
The details, including link case, will be given elsewhere.

The author wishes to express his thank to Professor Yoshikazu Yamaguchi 
for many helpful conversations.
He also thank to Professors Takahiro Kitayama, 
Takayuki Morifuji and Joan Porti for several comments. 

\section{Reidemeister torsions and twisted Alexander invariants}

Following \cite{porti2} and \cite{yamaguchi1}, we review some definitions and conventions 
in this section. 

Let $\mathbb F$ be a field and $C_*=(C_*,\bdd)$ 
a chain complex of finite dimensional $\mathbb F$-vector spaces: 
$$
0\to C_d\overset{\bdd}{\to}C_{d-1}\overset{\bdd}{\to}\cdots\overset{\bdd}{\to}C_0\to 0. 
$$
For each $i$, we denote by 
$B_i= {\rm Im}(C_{i+1}\overset{\bdd}{\to}C_i)$, 
$Z_i=\ker(C_{i}\overset{\bdd}{\to}C_{i-1})$, 
and 
the homology is denoted by 
$H_i=Z_i/B_i$. 
By the definition of $Z_i,\,B_i$ and $H_i$, 
we obtain the following exact sequence: 
\begin{align*}
0\to Z_i\to&\, C_i\overset{\bdd}{\to} B_{i-1}\to 0,\\ 
0\to B_i\to&\, Z_i\to H_{i}\to 0.
\end{align*}
Let $\widetilde{B}_{i-1}$ be a lift of $B_{i-1}$ to $C_{i}$, 
and $\widetilde{H}_{i}$ a lift of $H_i$ to $Z_{i}$. 
Then we can decompose $C_i$ as follows: 
\begin{align*}
C_i &= Z_{i}\oplus\widetilde{B}_{i-1}\\
    &= B_i\oplus\widetilde{H}_{i}\oplus\widetilde{B}_{i-1}.
\end{align*}  

Let $c^i$
be a basis for $C_i$ 
and 
${\bf c}$ the collection $\{c^i\}_{i\ge 0}$.
Similarly,
let 
$h^i$ be a basis for $H_i$, if nonzero, and 
${\bf h}$ the collection $\{h^i\}_{i\ge 0}$.
We choose $b^i$ a basis of $B_i$. 
Let $\widetilde{b}^{i-1}$ be a lift of $b^{i-1}$ to $C_{i}$, 
and $\widetilde{h}^{i}$ a lift of $h^{i}$ to $Z_i$, then 
we have a new basis $b^i\sqcup\widetilde{b}^{i-1}\sqcup\widetilde{h}^{i}$ of $C_i$, 
where $\sqcup$ means a disjoint union.
We denote by $[b^i,\widetilde{b}^{i-1},\widetilde{h}^{i}/c^i]$ 
the determinant of the transformation matrix from the basis $c^i$ 
to $b^i\sqcup\widetilde{b}^{i-1}\sqcup\widetilde{h}^{i}$. 

\begin{definition}\label{def:torsion}
The torsion of the chain complex $C_*$ with basis ${\bf c}$ and ${\bf h}$ 
for $H_i$ is: 
$$\text{tor}(C_*,{\bf c},{\bf h})
=
\prod^{d}_{i=0}[b^i,\widetilde{b}^{i-1},\widetilde{h}^{i}/c^i]^{(-1)^{i+1}}
\hspace{0.3cm}
\in\mathbb F^*/\{\pm 1\}
$$
\end{definition}

It is known that $\text{tor}(C_*,{\bf c},{\bf h})$ is independent of 
the choice of $b^i$ and the lifts $\widetilde{b}^{i-1}$ and $\widetilde{h}^{i}$. 

\begin{remark}\label{rmk:sign}
In \cite{menal-porti2}, they use $(-1)^i$ instead of $(-1)^{i+1}$ 
in Definition \ref{def:torsion}. 
Then the sign of the right hand side of the equation in Theorem 7.1 
in \cite{menal-porti2} 
becomes opposite. 
See Remark 2.2 and Theorem 4.5 
in \cite{porti2}.
\end{remark}

Let $W$ be a finite CW-complex, and 
$\rho:\pi_1(W,\ast)\to\text{SL}(n,\mathbb F)$
a representation of its fundemental group. 
Consider the chain complex of vector spaces 
$$C_*(W,\rho):=\mathbb F^n\otimes_{\rho}C_*(\widetilde{W};\mathbb Z)$$
where $C_*(\widetilde{W},\mathbb Z)$ denotes 
the simplicial complex of the universal covering of $W$ 
and $\otimes_{\rho}$ means that one takes the quotient 
of $\mathbb F^n\otimes_{\mathbb Z}C_*(\widetilde{W};\mathbb Z)$ 
by $\mathbb Z$-module generated by 
$$\rho(\gamma)^{-1}v\otimes c-v\otimes\gamma\cdot c.$$
Here, 
$v\in\mathbb F^n,\,\gamma\in \pi_1(W,\ast)$ 
and $c\in C_*(\widetilde{W};\mathbb Z)$. 
Namely, 
$$v\otimes \gamma\cdot c=\rho(\gamma)^{-1}v\otimes c 
\hspace{0.5cm}
\forall\gamma\in\pi_1(W,p).
$$
The boundary operator is defined by linearity 
and 
$\bdd(v\otimes c)=(\text{Id}\otimes\bdd)(v\otimes c)=v\otimes\bdd c$. 
We denote by $H_*(W,\rho)$ the homology of this complex. 

Let $\{v_1,\ldots, v_n\}$ be  a basis of $\mathbb F^n$ and
let $c^i_1,\ldots , c^i_{k_i}$ denote the set of $i$-dimensional cells of $W$.
We take a lift $\tilde{c}^i_j$ of the cell $c^i_j$ in $\widetilde{W}$. 
Then, for each $i$, 
$\tilde{c}^i=\{\tilde{c}^i_1,\ldots ,\tilde{c}^i_{k_i}\}$ is a basis of 
the $\mathbb Z[\pi_1(W)]$-module $C_i(\widetilde{W};\mathbb Z)$. 
Thus we have the following basis of $C_i(W,\rho)$: 
$$c^i=\{v_1\otimes \tilde{c}^i_1,v_2\otimes \tilde{c}^i_1,\ldots,v_n\otimes \tilde{c}^i_{k_i}\}.$$ 
Suppose $H_i(W,\rho)\neq 0$, 
and $h^i$ be a basis of $H_i(W;\rho)$.
We denote by ${\bf h}$ the basis 
$\{h^0,\ldots, h^{\text{dim} W}\}$ of $H_*(W,\rho)$.
Then $\text{tor}(C_*(W,\rho),{\bf c},{\bf h})\,(\in \mathbb F^*/\{\pm 1\})$ 
is well defined. 
Note that it does not depend on the lifts of the cells $\tilde{c}^i$ 
since $\det\rho=1$. 
Further, if the Euler characteristic of $W$ is equal to zero 
(e.g. the case that $W$ corresponds to a knot exterior), 
it does not depend on the choice of a basis $\{v_1,\ldots, v_n\}$
(cf. Lemma 2.4.2 \cite{yamaguchi1}).
\begin{remark}
The Reidemeister torsion is independent of the choice of a base point 
$b$ of the fundamental group $\pi_1(W,\ast)$.
Furthermore, 
it is known that the Reidemeister torsion is an invariant 
under subdivision of the cell decomposition of $W$ with $\rho$-coefficients 
up to factor $\pm 1$. 
\end{remark}
\begin{remark}
Let $K$ be a knot in the 3-sphere $S^3$ 
and $M_K=S^3-\inter N(K)$.
We denote by $G(K)$ the fundamental group of $M_K$.
From the result of Waldhausen \cite{waldhausen}, the Whitehead group $\text{Wh}(G(K))$ 
is trivial. 
In such case, the Reidemeister torsion does not depend on 
the choice of its CW-structure.
Suppose $H_{*}(M_K,\rho)=0$. 
Then the Reidemeister torsion does not depend on ${\bf h}=\emptyset$. 
In this case we denote by $\text{tor}(M_K,\rho)$
the Reidemeister torsion.
\end{remark}

Let $\alpha$ be a surjective homomorphsim from $\pi_1(W,\ast)$ to 
the multiplicative group $\langle t\rangle$.
Instead of a representation $\rho:\pi_1(W,\ast)\to\text{SL}(n,\mathbb F)$, 
consider the twisted representation:
$$\alpha\otimes\rho:\pi_1(W,\ast)\to \text{GL}(\mathbb F(t)),$$
where $\mathbb F(t)$ is the filed of franction of the polynomial ring 
$\mathbb F[t]$.
By the same method as above, 
we can define 
$\text{tor}(C_*(W,\alpha\otimes\rho),{\bf 1}\otimes{\bf c},{\bf h})
\,(\in \mathbb F^*(t)/\{\pm t^{n\mathbb Z}\})$. 
As the determinant is not one, 
there is an independency factor $t^{nm}$, for some integer $m$.
More preciuosly, we define:
$$
C_*(W,\alpha\otimes\rho)
=
\mathbb F(t)\otimes_{\mathbb F}\mathbb F^n\otimes_{\rho}C_*(\widetilde{W};\mathbb Z),
$$
where the action is given by 
$f\otimes v\otimes(\gamma\cdot c)
=
f\cdot t^{\alpha(\gamma)}\otimes\rho(\gamma)^{-1}v\otimes c
$
for $\gamma\in\pi_1(W,p)$.
The boundary operator is defined by linearity and 
$\bdd(f\otimes v\otimes c)=f\otimes v\otimes \bdd c$.

Kitano (\cite{kitano}) investigated the relationship between the Reidemeister torsions
and the twisted Alexander invariants for knots. Namely, he proved that:
\begin{theorem}[\cite{kitano}]\label{thm:kitano}
Let $K$ be a knot in the 3-sphere $S^3$ 
and $M_K=S^3-\inter N(K)$.
Suppose $\rho$ is a non-trivial representation such that $H_*(M_K,\rho)=0$. 
Then, $H_*(M_K,\alpha\otimes\rho)=0$ and 
${\rm tor}(M_K,\alpha\otimes\rho)
=
\Delta_{K,\rho}(t)$, 
where $\Delta_{K,\rho}(t)$ is the twisted Alexander invariant.
\end{theorem}
See also Theorem 2.13 in \cite{porti2}.
The twisted Alexander invariant can be computed using the Fox calculus 
\cite{kirk-livingston, kitano, wada}.

\section{Representations of the fundamental groups of hyperbolic 3-manifolds}

Let $M$ be an oriented, complete, hyperbolic 3-manifold of finite volume. 
Then $M$ has the holonomy representation: 
${\rm Hol}_M:\pi_1(M,\ast)\to {\rm Isom^+ \mathbb H^3},$
where ${\rm Isom}^+\mathbb H^3$
is the orientation preserving isometry group of hyperbolic 3-space $\mathbb H^3$. 
Using the upper half-space model, ${\rm Isom}^+\mathbb H^3$ is identified with PSL$(2,\mathbb C)
=\text{SL}(2,\mathbb C)/\{\pm 1\}$. 
It is known that ${\rm Hol}_M$ can be lifted to SL$(2,\mathbb C)$, and 
such lifts are in canonical one-to-one correspondence with spin structures on $M$.
Thus, attached to a fixed spin structure $\eta$ on $M$, we get a representation:
$$
\text{Hol}_{(M,\eta)}: \pi_1((M,\eta),\ast)\to\text{SL}(2,\mathbb C).$$

Let $W$ be a finite CW-complex and $\rho$ a representation of $\pi_1(W,\ast)$ 
to $\text{SL}(2,\mathbb C)$. 
Then the pair $(\mathbb C^2,\rho)$ is an $\text{SL}(2,\mathbb C)$-representation 
of $\pi_1(W,\ast)$ by the standrd action $\text{SL}(2,\mathbb C)$ to $\mathbb C^2$.
It is known that the pair of the symmetric product $\text{Sym}^{n-1}(\mathbb C^2)$ 
and the induced action by $\text{SL}(2,\mathbb C)$ gives an $n$-dimensional 
irreducible representation of $\text{SL}(2,\mathbb C)$. 
More precisely, 
let $V_n$ be the vector space of homogeneous polynomials on $\mathbb C^2$ 
with degree $n-1$, that is,
$$
V_n
=\text{span}_{\mathbb C}
\langle
x^{n-1}, x^{n-2}y,\ldots,xy^{n-2},y^{n-1}
\rangle.
$$
Then 
the symmetric product $\text{Sym}^{n-1}(\mathbb C^2)$ can be identified 
with $V_n$ and the action of $A\in\text{SL}(2,\mathbb C)$ 
is expressed as 
$$
A\cdot p
\begin{pmatrix}
x\\
y
\end{pmatrix}
=
p\big(A^{-1}
\begin{pmatrix}
x\\
y
\end{pmatrix}
\big)
$$
where 
$\displaystyle{p
\begin{pmatrix}
x\\
y
\end{pmatrix}}$ 
is a homogeneous polynomial 
and the right hand side is determined by the action of $A^{-1}$ 
on the column vector as a matrix multiplication.
We denote by $(V_n,\sigma_n)$ the representation given by this action 
of $\text{SL}(2,\mathbb C)$ where $\sigma_n$ means the homomorphism
from $\text{SL}(2,\mathbb C)$ to $\text{GL}(V_n)$. 
It is known that 
each representation $(V_n,\sigma_n)$ turns into an irreducible 
$\text{SL}(n,\mathbb C)$-representation of $\text{SL}(2,\mathbb C)$ 
and that 
every irreducible $n$-dimensional representation of $\text{SL}(2,\mathbb C)$ 
is equivalent to $(V_n,\sigma_n)$. 
Composing $\text{Hol}_{(M,\eta)}$ with $\sigma_n$, 
we obtain the following representation:
$$
\rho_n:\pi_1((M,\eta),\ast)\to {\rm SL}(n,\mathbb C).
$$

In the following section, 
we will discuss Reidemeister torsions associated with this representation $\rho_n$. 
Note that there are several computations of the Reidemeister torsions 
associated with $\sigma_{2k}$ in \cite{yamaguchi2, yamaguchi3}.

\section{The results of Menal-Ferrer and Porti}
In this note, we focus on a knot complement. 
We introduce the results of Menal-Ferrer and Porti 
\cite{menal-porti1,menal-porti2}
in this setting. 

Let $K$ be a hyperbolic knot in the 3-sphere $S^3$, 
that is,  
$S^3-K$ is an oriented, complete, finite-volume hyperbolic manifold
with only one cusp. 
Then, $S^3-K$ may be regarded as the interior of a compact manifold $M_K$ 
such that $\bdd M_K=T$ where $T$ is homeomorphic to a torus $T^2$.
In what follows, we consider the compact manifold $M_K$ instead of $S^3-K$.

By  Corollary 3.7 in \cite{menal-porti1}, we have that
$\dim_{\mathbb C}H^i(M_K,\rho_n)=0$ $(i=0,1,2)$
if $n$ is even, and that
$\dim_{\mathbb C}H^0(M_K,\rho_n)=0,$ 
$\dim_{\mathbb C}H^1(M_K,\rho_n)=\dim_{\mathbb C}H^2(M_K,\rho_n)=1$ 
if $n$ is odd.
Further, in \cite{menal-porti2}, Menal-Ferrer and Porti proved the following.
(Note that Poincar\'e duality with coefficients in $\rho_n$ holds 
(Corollary 3.7 in \cite{menal-porti2}.))

\begin{proposition}[Proposition 4.6 in \cite{menal-porti2}]\label{prop:homology}
Suppose that $H_*(T;\rho_n)\neq0$.
Let $G<\pi_1(M_K,\ast)$ be some fixed realization of the fundamental group of $T$ 
as a subgroup of $\pi_1(M_K,\ast)$. 
Choose a non-trivial cycle $\theta\in H_1(T;\mathbb Z)$, 
and a non-trivial vector $v\in V_n$ fixed by $\rho_n(G)$.
Then the following holds:
\begin{enumerate}
\item
A basis for $H_1(M_K,\rho_n)$ is given by $i_*([v\otimes\widetilde\theta])$.
\item
A basis for $H_2(M_K,\rho_n)$ is given by $i_*([v\otimes \widetilde T])$.
\end{enumerate}
Here, $i:T\hookrightarrow M_K$ denotes the inclusion.
\end{proposition}
Set 
$h^1=i_*([v\otimes\widetilde\theta]), 
 h^2=i_*([v\otimes \widetilde T])$, 
and
${\bf h}=\{h^1,h^2\}$.
On the other hand,
Menal-Ferrer and Porti (Theorem 0.2 in \cite{menal-porti1}) 
proved that $H^*(M_K,\rho_{2k})=0$ for $k\ge 1$.
Therefore, we may define the following quotients.
\begin{align*}
\mathcal T_{2k+1}(M_K,\eta)&:=
\frac{\text{tor}(M_K,\rho_{2k+1},{\bf h})}{\text{tor}(M_K,\rho_3,{\bf h})}
\hspace{0.5cm}\in\mathbb C^*/\{\pm 1\}\\
\mathcal T_{2k}(M_K,\eta)&:=
\frac{\text{tor}(M_K,\rho_{2k})}{\text{tor}(M_K,\rho_2)}
\hspace{0.5cm}\in\mathbb C^*/\{\pm 1\}
\end{align*}
The quantity $\mathcal T_{2k+1}$ is independent of the spin structure 
because of the fact that an odd-dimensional irreducible complex representation 
of $\text{SL}(2,\mathbb C)$ factors through $\text{PSL}(2,\mathbb C)$.
Since $S^3-K$ has only one cusp, then all spin structures on $M_K$ 
are acyclic (Corollary 3.4 in \cite{menal-porti2}). 
This means that $\mathcal T_{2k}$ is also independent of the spin structure
(Theorem 7.1 \cite{menal-porti2}).
Thus it is not necessary to consider a spin structure on $M_K$
in our setting.
Hence, 
the above definition may be simplified to the following form 
deleting $\eta$.
\begin{definition}
\begin{align*}
\mathcal T_{2k+1}(M_K)&:=
\frac{\text{tor}(M_K,\rho_{2k+1},{\bf h})}{\text{tor}(M_K,\rho_3,{\bf h})}
\hspace{0.5cm}\in\mathbb C^*/\{\pm 1\}\\
\mathcal T_{2k}(M_K)&:=
\frac{\text{tor}(M_K,\rho_{2k})}{\text{tor}(M_K,\rho_2)}
\hspace{0.5cm}\in\mathbb C^*/\{\pm 1\}
\end{align*}
\end{definition}
Note that it is proved  
that the quotient 
is independent of the choices ${\bf h}$
(Proposition 4.2 in \cite{menal-porti2}). 
Then, we can reduce Theorem 7.1 in \cite{menal-porti2} to 
the following statement: 

\begin{theorem}[Theorem 7.1 in \cite{menal-porti2}]\label{thm:portivolume}

$$
\lim_{k\to\infty}\frac{\log|\mathcal T_{2k+1}(M_K)|}{(2k+1)^2}
=
\lim_{k\to\infty}\frac{\log|\mathcal T_{2k}(M_K)|}{(2k)^2}=\frac{{\rm Vol}(K)}{4\pi}.
$$
\end{theorem}

As in Remark \ref{rmk:sign}, the sign of the right hand side is plus.

\section{Proof of Theorem \ref{thm:main}}

\noindent {\bf Case 1.} Even-dimensional representation $\rho_{2k}$ case.

By Theorem 0.2 in \cite{menal-porti1}, 
$H^*(M_K,\rho_{2k})=0$ for $k\ge 1$.
Then, by Theorem \ref{thm:kitano},  
we can prove that 
$\text{tor}(M_K,\rho_{2k})
=
\text{tor}(M_K,\alpha\otimes\rho_{2k})|_{t=1}
=
\Delta_{K,\rho_{2k}}(1)
$ 
from the map at the chain level 
$C_*(M_K,\alpha\otimes\rho_{2k})\to C_*(M_K,\rho_{2k})$ induced by evaluation $t=1$.
Then, we have: 
$$\mathcal T_{2k}(M_K)
=
\frac{\text{tor}(M_K,\rho_{2k})}{\text{tor}(M_K,\rho_2)}
=
\frac{\Delta_{K,\rho_{2k}}(1)}{\Delta_{K,\rho_2}(1)}
=
\mathcal A_{K,2k}(1).
$$
Hence we have done in the case of $\rho_{2k}$ in Theorem \ref{thm:main}: 
$\displaystyle{\lim_{k\to\infty}\frac{\log|\mathcal A_{K,2k}(1)|}{(2k)^2}=\frac{{\rm Vol}(K)}{4\pi}}$
by Theorem \ref{thm:portivolume}.

\medskip

\noindent{\bf Case 2}.
Odd-dimensional representation $\rho_{2k+1}$ case.

Although the idea of the proof is the same as 
Yamaguchi's one in \cite{yamaguchi0, yamaguchi1}, 
I think it is worth outlining it here for the convenience of readers.
He investigated the case of the adjoint representation of 
$\text{SL}(2,\mathbb C)$, 
which is essentially equivalent to $\rho_3$ in our setting.

The homology group 
$H_*(M_K;\mathbb Z)=H_0(M_K;\mathbb Z)\oplus H_1(M_K;\mathbb Z)$ 
has the basis $\{[p], [\mu]\}$, 
where $[p]$ is the homology class of a point 
and $[\mu]$ is that of the meridian of $K$.
Further, 
$H_1(\bdd M_K;\mathbb Z)$ has the basis $\{[\mu],[\lambda]\}$, 
where $[\lambda]$ is the homology class of a longitude of $K$. 
By Proposition \ref{prop:homology}, 
we may define 
$h^1=i_*([v\otimes\widetilde\lambda]),\, 
 h^2=i_*([v\otimes\widetilde T])$
and ${\bf h}=\{h^1, h^2\}$.

It is known that $M_K$ collapses to a 2-dimensional CW-complex $W$
with only one vertex. 
We call $\varphi$ this deformation. 
Thus $M_K$ is simple homotopy equivalent to $W$. 
It is enough to prove the theorem for $W$ 
since a Reidemeister torsion is a simple homotopy invariant. 

By Proposition 3.5 in \cite{kirk-livingston},
we have 
$H_0(W,\alpha\otimes\rho_{2k+1})=0$. 
Further, we have the next lemma by
the same argument as Proposition 7 in \cite{yamaguchi0} 
or Proposition 3.1.1 in \cite{yamaguchi1}. 

\begin{lemma}\label{lem:acyclic}
For $*=1,\,2$, we have: $H_*(M_K,\alpha\otimes\rho_{2k+1})=0$.
\end{lemma}

\begin{proposition}\label{prop:yamaguchi}
${\rm tor}(M_K,\alpha\otimes\rho_{2k+1})$ has a simple zero at $t=1$. 
Moreover the following holds:
$$
{\rm tor}(M_K,\rho_{2k+1},{\bf h})
=
\lim_{t\to 1}\frac{{\rm tor}(M_K,\alpha\otimes\rho_{2k+1})}{t-1}.
$$
\end{proposition}
\begin{proof}
We define the subchain complex $C'_*(W,\rho_{2k+1})$ of the chain complex $C_*(W,\rho_{2k+1})$ by
$$
C'_2(W,\rho_{2k+1})=\text{span}_{\mathbb C}\langle v\otimes\widetilde{\varphi(T)}\rangle,\,
\hspace{0.5cm}
C'_1(W,\rho_{2k+1})=\text{span}_{\mathbb C}\langle v\otimes\widetilde{\varphi(\lambda)}\rangle
$$
and 
$C'_i(W,\rho_{2k+1})=0\,\,(i\neq 1,2)$. 
Note that $v$ is fixed by $\rho_{2k+1}(G)$, 
and the boundary operators of $C'_*(W,\rho_{2k+1})$ are zero 
by the definition.
The modules of this subchain complex are 
lifts of homology groups $H_*(W,\rho_{2k+1})$.
Similarly, 
we define the subcomplex $C'_*(W,\alpha\otimes\rho_{2k+1})$ 
of $C_*(W,\alpha\otimes\rho_{2k+1})$ by 
$$
C'_2(W,\alpha\otimes\rho_{2k+1})=\text{span}_{\mathbb C(t)}\langle 1\otimes v\otimes\widetilde{\varphi(T)}\rangle,\,
\hspace{0.3cm}
C'_1(W,\alpha\otimes\rho_{2k+1})=\text{span}_{\mathbb C(t)}\langle 1\otimes v\otimes\widetilde{\varphi(\lambda)}\rangle
$$
and 
$C'_i(W,\alpha\otimes\rho_{2k+1})=0$ for $i\neq 1,2$. 
Since $v$ is an invariant vector of $\rho_{2k+1}(G)$, 
we have: 
\begin{align*}
\bdd(1\otimes v\otimes\widetilde{\varphi(T)})
&=1\otimes v \otimes\bdd(\widetilde{\varphi(T)})\\
&=1\otimes v\otimes(\mu\cdot\widetilde{\varphi(\lambda)})
	-1\otimes v\otimes\widetilde{\varphi(\lambda)}\\
&=t\otimes\rho_{2k+1}^{-1}(\mu)v\otimes\widetilde{\varphi(\lambda)}
	-1\otimes v\otimes\widetilde{\varphi(\lambda)}\\
&=t\otimes v\otimes\widetilde{\varphi(\lambda)}
	-1\otimes v\otimes\widetilde{\varphi(\lambda)}\\
&=(t-1)(1\otimes v\otimes\widetilde{\varphi(\lambda)})
\end{align*}
Thus the boundary operators of $C'_*(W,\alpha\otimes\rho_{2k+1})$ 
is given by 
$$
0\to C'_2(W,\alpha\otimes\rho_{2k+1})\overset{t-1}{\longrightarrow}
C'_1(W,\alpha\otimes\rho_{2k+1})\to0.
$$
This means that the homology of
$C'_*(W,\alpha\otimes\rho_{2k+1})$ is zero. 

By the definition,  
the chain complex $C'_*(W,\rho_{2k+1})$ has the natural basis:
$$
{\bf c'}=\{v\otimes\widetilde{\varphi(T)},\,v\otimes\widetilde{\varphi(\lambda)}\}.
$$
Let $C''_*(W,\rho_{2k+1})$ be the quotient of 
$C_*(W,\rho_{2k+1})$ by $C'_*(W,\rho_{2k+1})$, 
${\bf c''}$ a basis of $C''_*(W,\rho_{2k+1})$,  
and 
${\bf \bar{c}''}$ a lift of ${\bf c''}$ to $C_*(W,\rho_{2k+1})$. 
By Lemma \ref{lem:acyclic}, we can apply Proposition 3.3.1 in \cite{yamaguchi1}
to this setting, then we have:
$$
\lim_{t\to 1}
\frac{{\rm tor}(C_*(W,\alpha\otimes\rho_{2k+1}), {\bf 1}\otimes{\bf c'}\sqcup{\bf 1}\otimes{\bf\bar{c}''})}
{{\rm tor}(C'_*(W,\alpha\otimes\rho_{2k+1}),{\bf 1}\otimes{\bf c'})}
=
{\rm tor}(C_*(W,\rho_{2k+1}),{\bf c'}\sqcup{\bf\bar{c}''},{\bf h}).
$$
By the calculation above, 
we have ${\rm tor}(C'_*(W,\alpha\otimes\rho_{2k+1}),{\bf 1}\otimes {\bf c'})=t-1$, 
thus we have this proposition.
\end{proof}

\noindent{\bf Proof of Theorem \ref{thm:main}}.

By Theorem \ref{thm:kitano} and Lemma \ref{lem:acyclic}, we have 
$\text{tor}(M_K,\alpha\otimes\rho_{2k+1})=\Delta_{K,\rho_{2k+1}}(t).$
We also have 
$\Delta_{K,\rho_{2k+1}}(t)=(t-1)\tilde{\Delta}_{K,\rho_{2k+1}}(t)$ 
and
$\text{tor}(M_K,\rho_{2k+1},{\bf h})=\tilde{\Delta}_{K,\rho_{2k+1}}(1)$
by Proposition \ref{prop:yamaguchi}, 
where $\tilde{\Delta}_{K,\rho_{2k+1}}(t)$ is a rational function.
Then, 
\begin{align*}
\mathcal A_{K,2k+1}(1)
=
\frac{\tilde{\Delta}_{K,\rho_{2k+1}}(1)}{\tilde{\Delta}_{K,\rho_{3}}(1)}
=
\frac{\text{tor}(M_K,\rho_{2k+1},{\bf h})}{\text{tor}(M_K,\rho_3,{\bf h})}
=
\mathcal T_{2k+1}(M_K).
\end{align*}
Thus we have Theorem \ref{thm:main} by Theorem \ref{thm:portivolume}.
\begin{flushright}
$\square$
\end{flushright}

\section{Some calculations on the figure eight knot complement}

Let $K$ be the figure eight knot $4_1$. 
Note that it is known that the volume of $K$ is $2.02988\cdots$.
The knot group $G(K)$ has the following presentation:
$$G(K)=
\langle 
a,b~|~ab^{-1}a^{-1}ba=bab^{-1}a^{-1}b
\rangle,
$$
where $a$ and $b$ correspond to the meridians of $K$.
Consider the representation of this fundamental group: 
$$
\rho(a)=
\begin{pmatrix}
1 & 1 \\
0 & 1
\end{pmatrix},\,
\rho(b)=
\begin{pmatrix}
\hphantom{-}1 & 0 \\
-u & 1
\end{pmatrix},
$$
where $u$ is a complex value satisfying $u^2+u+1=0$. 
This representation is the holonomy representation 
of $G(K)$. 
By the definition, 
we have 
$p\big(\rho(a)^{-1}
\begin{pmatrix}
x\\
y
\end{pmatrix}
\big)
=
p
\begin{pmatrix}
x-y\\
y
\end{pmatrix}
$, 
and
$(x-y)^2=x^2-2xy+y^2,\,
 (x-y)y=xy-y^2.
$
Hence, we have:
$$
\rho_3(a)=
\begin{pmatrix}
\hphantom{-}1 & \hphantom{-}0 & \hphantom{-}0 \\
-2 & \hphantom{-}1 & \hphantom{-}0\\
\hphantom{-}1 & -1 & \hphantom{-}1
\end{pmatrix}.
$$
By the same calculations, we have :
$$
\rho_3(b)=
\begin{pmatrix}
1 & u & u^2 \\
0 & 1 & 2u \\
0 & 0 & 1
\end{pmatrix}
,\,
\rho_4(a)=
\begin{pmatrix}
\hphantom{-}1 & \hphantom{-}0 & \hphantom{-}0 & 0 \\
-3 & \hphantom{-}1 & \hphantom{-}0 & 0 \\
\hphantom{-}3 & -2 & \hphantom{-}1 & 0\\
-1 & \hphantom{-}1 & -1 & 1
\end{pmatrix},\,
\rho_4(b)=
\begin{pmatrix}
1 & u & u^2 & u^3 \\
0 & 1 & 2u & 3u^2\\
0 & 0 & 1 & 3u\\
0 & 0 & 0 & 1
\end{pmatrix},
\cdots
$$
Set 
$A=\rho_2(a)={}^t\rho(a)^{-1}=
\begin{pmatrix}
\hphantom{-}1 & 0\\
-1 & 1
\end{pmatrix}$
and 
$B=
\rho_2(b)={}^t\rho(b)^{-1}=
\begin{pmatrix}
1 & u\\
0 & 1
\end{pmatrix}$.
Via Fox's calculus for $G(K)$, 
we obtain the denominator of 
$\Delta_{K,\rho_2}(t)=\det(tB-I)=(t-1)^2$. 
On the other hand, 
the numerator of 
$\Delta_{K,\rho_2}(t)=\det(I-t^{-1}AB^{-1}A^{-1}+
AB^{-1}A^{-1}B-tB+BAB^{-1}A^{-1})
=\frac{1}{t^2}(t-1)^2(t^2-4t+1).$
Here we use the value 
$u=\frac{-1+\sqrt{-3}}{2}$. 
Continuing in this way, we have obtained the following 
data. 
$$
\Delta_{K,\rho_2}(t)=
\frac{1}{t^2}(t^2-4t+1),\,
\Delta_{K,\rho_3}(t)=
-\frac{1}{t^3}(t-1)(t^2-5t+1)
$$
$$
\Delta_{K,\rho_4}(t)=
\frac{1}{t^4}(t^2-4t+1)^2,\,
\Delta_{K,\rho_5}(t)=
-\frac{1}{t^5}(t-1)(t^4-9t^3+44t^2-9t+1),\,
$$
\begin{align*}
\frac{4\pi \log\mathcal |\mathcal A_{K,4}(t)|}{4^2}
&=\frac{\pi\log|t^2-4t+1|}{4}\overset{t=1}{\longrightarrow}\frac{\pi\log 2}{4}\approx 0.544397\cdots \\
\frac{4\pi \log\mathcal |\mathcal A_{K,5}(t)|}{5^2}
&=\frac{4\pi\log|\frac{t^4-9t^3+44t^2-9t+1}{t^2-5t+1}|}{5^2}
\overset{t=1}{\longrightarrow} \frac{4\pi\log\frac{28}{3}}{5^2}\approx 1.12273\cdots
\end{align*}

\begin{table}[htb]
  \begin{tabular}{|c|c||c|c|} \hline
 $n$ & $\frac{4\pi\log|\mathcal A_{K,n}(1)|}{n^2}$ & $n$ & $\frac{4\pi\log|\mathcal A_{K,n}(1)|}{n^2}$ \\ \hline \hline
 $6$ & $1.35850\cdots$ & $7$ & $1.58331\cdots$ \\ \hline
 $8$ & $1.66441\cdots$ & $9$ & $1.76436\cdots$ \\ \hline
 $10$ & $1.79618\cdots$ & $11$ & $1.85105\cdots$ \\ \hline
 $12$ & $1.86678\cdots$ & $13$ & $1.90158\cdots$ \\ \hline
 $14$ & $1.91009\cdots$ & $15$ & $1.93361\cdots$ \\ \hline
 \end{tabular}
\end{table}
These calculations were done by using Wolfram Mathematica.

\bibliographystyle{amsplain}

\end{document}